\documentclass{amsart}
\usepackage[utf8]{inputenc}
\usepackage{amsmath}
\usepackage{amscd}
\usepackage{amssymb}
\usepackage{graphicx}
\usepackage{graphics}
\usepackage{latexsym}
\usepackage{cleveref}
\usepackage{amsthm}
\usepackage[dvipsnames]{xcolor}

\usepackage{tikz}
\usetikzlibrary{decorations.markings,arrows,automata,arrows,backgrounds,snakes,arrows.meta,patterns}
\usepackage{amssymb,amsmath}
\usetikzlibrary{intersections}
\usepackage{listings}
\usepackage{rotating}
\usepackage{url}

\input xy
\xyoption{all}

\newtheorem{theorem}{Theorem}[section]
\newtheorem{lemma}[theorem]{Lemma}
\newtheorem{proposition}[theorem]{Proposition}

\newtheorem{cor}[theorem]{Corollary}

\theoremstyle{definition}
\newtheorem{definition}[theorem]{Definition}
\newtheorem{remark}[theorem]{Remark}

\newcounter{example}[section]

\newcommand{\zr}{{\mathbb R}}

\newcommand{\zz}{{\mathbb Z}}

\newcommand{\zm}{{\mathbb M}}

\newcommand{\fk}{\textrm{F}^+\textrm{K}}
\newcommand{\fkmath}{\textrm{\em F}^+\textrm{\em K}}

\begin{document}

\title{Discrete Morse Theory on $\Omega S^2$}
\author{Lacey Johnson}
\email{ljohnson@ut.edu}
\address{Department  of Mathematics, University of Tampa, Tampa, FL 33606}
\author{Kevin Knudson}
\email{kknudson@ufl.edu }
\address{Department  of Mathematics,  University of Florida, Gainesville, FL 32611}
\date{\today}

\begin{abstract}
A classical result in Morse theory is the determination of the homotopy type of the loop space of a manifold. In this paper, we study this result through the lens of discrete Morse theory. This requires a suitable simplicial model for the loop space. Here, we use Milnor's $\fk$ construction to model the loop space of the sphere $S^2$, describe a discrete gradient on it, and identify a collection of critical cells. We also compute the action of the boundary operator in the Morse complex on these critical cells, showing that they are potential homology generators. A careful analysis allows us to recover the calculation of the first homology of $\Omega S^2$.
\end{abstract}

\maketitle

\section{Introduction}
Suppose $M$ is a Riemannian manifold and let $p,q$ be points in $M$ (not necessarily distinct). Denote by $\Omega=\Omega(M,p,q)$ the space of piecewise smooth paths from $p$ to $q$. Given such a path $\gamma:[0,1]\to M$, the {\em energy} of $\gamma$ is defined as
$$E(\gamma) = \int_0^1 \biggl|\biggl|\frac{d\gamma}{dt}\biggr|\biggr|^2\,dt.$$ The critical points of the function $E$ are precisely the geodesics and one can define a bilinear operator (the {\em Hessian}) $E_{\ast\ast}$ on the tangent space $T\Omega_\omega$ at a critical point $\omega$. The {\em index} of this critical point is then defined to be the maximum dimension of a subspace of $T\Omega_\omega$ on which $E_{\ast\ast}$ is negative definite. Morse proved that this number is always finite and equals the number of points  that are conjugate to $\omega(0)$ along $\omega$. 

One of the most spectacular applications of Morse theory is the following. If $M$ is a complete Riemannian manifold and if $p$ and $q$ are two points that are not conjugate along any geodesic in $M$, then $\Omega(M,p,q)$ has the homotopy type of a countable CW-complex with one cell of dimension $\lambda$ for each geodesic from $p$ to $q$ of index $\lambda$. Indeed, Milnor \cite{milnor} calls this the ``Fundamental Theorem of Morse Theory." As a consequence, one deduces that the loop space $\Omega S^n$ of the $n$-sphere has the homotopy type of a CW-complex with one cell in each dimension $0$, $(n-1)$, $2(n-1)$, $3(n-1)$,... These cells correspond to the geodesics traced by great circles on $S^n$, the $k(n-1)$-cell coming from traversing the great circle $k$ times. If $n>2$ this provides an easy computation of the homology of $\Omega S^n$. The case $n=2$ requires more care, but it can be done.

Forman introduced {\em discrete Morse theory} on cell complexes in \cite{forman}. This is a combinatorial version of the classical smooth Morse theory and it has proven to be extremely useful in topology, combinatorics, homological algebra, and other fields. We will review the basics of this theory in Section \ref{sec:dmt} below, but the upshot is that given a CW-complex $X$ one can define a discrete Morse function as a real-valued map on the set of cells of $X$, denoted $f:X\to \zr$. A cell is called critical if certain combinatorial conditions hold and then one obtains an analogue of the result from smooth Morse theory: the complex $X$ has the homotopy type of a CW-complex with one cell for each critical cell of $f$. Associated to such a function is a discrete gradient vector field $V$, which is a partial matching on the set of cells of $X$. In fact, one can focus on the matching--using it one modifies the Hasse diagram of $X$, reversing an edge when a pair of cells is in the matching, and the resulting directed graph is acyclic. It therefore suffices to find such an acyclic partial matching on the cells (a {\em Morse matching}), and that is the approach we will take here. Moreover, there is a chain complex $\zm_\bullet(X,V)$ built from the critical cells of $V$, with a boundary map defined in terms of discrete gradient paths in $X$, whose homology agrees with the singular homology of $X$.

In this paper we attempt to merge these two ideas by studying discrete Morse theory on loop spaces, with a focus on the simplest nontrivial example: $\Omega S^2$. While this space has the homotopy type of a CW-complex (with one cell in each nonnegative dimension, as noted above), it does not come equipped with a cell structure. We must therefore find a simplicial model for it and for that we turn to Milnor's $\fk$ construction \cite{milnor2}. This is a simplicial version of James's reduced product of a CW-complex $X$ \cite{james}, which is homotopy equivalent to the loop space of the suspension of $X$, $\Omega \Sigma X$. Taking $K$ to be a simplicial set modeling the circle $S^1$, we obtain a simplicial model for $\Omega S^2 = \Omega\Sigma(S^1)$; we describe this in Section \ref{sec:fk}. After a discussion of the basics of discrete Morse theory on simplicial sets in Section \ref{sec:dmt}, we describe a Morse matching on our simplicial model of $\Omega S^2$ in Section \ref{sec:loopss2}. In Section \ref{sec:crit}  we describe some critical cells of the matching and  in Section \ref{sec:homology} we compute the action of the boundary map on these cells, suggesting that they are homologically significant. The dream, of course, would be to find a Morse matching on $\fk$ with a single critical cell of each dimension, but for various reasons mentioned below, that is not possible. Still, the group of critical cells we identify are a nice collection to work with.

While we focus on a particular example here, the techniques are applicable to more general loop spaces, provided one is willing to pass to a suspension. In particular, the loop spaces of higher dimensional spheres should be accessible through very similar calculations.

\subsection*{Some history} This project has its origins in conversations the second author had with Daniel Biss in 2005. After a couple months' work with little progress we both turned our attention to other things (in Daniel's case, that meant state and local politics; he is currently the mayor of Evanston, IL). The second author then suggested this as a dissertation project to the first author, and a first iteration became her doctoral thesis \cite{lj}. The pairing described in that version, which we called the ``Max Pairing" was a nice construction, but it has cycles (i.e., it is not a Morse matching). After learning of Lampret's introduction of {\em steepness pairings} \cite{lampret}, we eventually decided to revisit this project. This paper is the result.

\section{The $\fk$ construction and $\Omega S^2$}\label{sec:fk}
Suppose $K$ is a simplicial set and denote the face and degeneracy maps by $d_i:K_n\to K_{n-1}$, $s_i:K_n\to K_{n+1}$, respectively. For general details about simplicial sets we refer the reader to \cite{may}. We will need to make use of the following relation on the degeneracies: $s_is_j = s_{j+1}s_i$, $i\le j$. We assume that $K$ has a distinguished base point $b_0\in K_0$ and we set $b_n=s_0^n(b_0)\in K_n$.

\begin{definition}\label{def:fk} For each $n\ge 0$, let $\fk_n$ be the free monoid generated by the elements of $K_n$, with the single relation $b_n=1$. Then $\fk$ is the simplicial monoid whose set of $n$-simplices is the monoid $\fk_n$, with face and degeneracy operators the unique homomorphisms carrying the generators $k_n\in K_n$ to $d_ik_n$ and $s_ik_n$, respectively.
\end{definition}

\begin{remark} Milnor \cite{milnor2} also studies $\mathrm{FK}$,  the analogous simplicial object whose $n$-simplices are the free {\em group} on the elements of $K_n$. He proves that $\mathrm{FK}$ is homotopy equivalent to $\fk$, so we are free to use either one. For our purposes, the monoid is more convenient, so we will use it here. Moreover, $\mathrm{F}^+\mathrm{K}$ is the direct simplicial analogue of James's reduced product construction \cite{james}.
\end{remark}

\begin{theorem}[\cite{milnor2}, Theorem 1]\label{thm:fkloop} There is a homotopy equivalence $|\fkmath|\simeq \Omega\Sigma |K|$.
\end{theorem}

Now consider the simplicial set $K$ having a unique $0$-simplex $e$ and a single $1$-simplex $y$, along with all the degeneracies of these. Thus, $K_1$ has elements $y$ and $s_0e$, $K_2$ has elements $s_0y, s_1y$, and the iterated degeneracies of the $0$-simplex $e$, and so on. The geometric realization $|K|$ is homeomorphic to the circle $S^1$. It follows that $\fk$ has the homotopy type of the loop space of the suspension of $S^1$; i.e., $|\fk|\simeq\Omega S^2$. 

Let us describe the simplicial monoid $\fk$ in detail. First note that $\fk_0$ is the free monoid on the element $e$, with the relation that $e=1$; thus, it is the trivial monoid. Then $\fk_1$ is the free monoid on $y$ and $s_0e$, modulo the relation that $s_0e=1$; thus it is the free monoid on one generator $y$. The monoid $\fk_2$ then has generators $s_0y$ and $s_1y$ (the iterated degeneracies of $e$ all vanish since these maps are monoid homomorphisms and $e$ represents the identity element in $\fk_0$; we therefore use $e$ to denote the identity element in each $\fk_n$). The monoid $\fk_3$ is generated by $s_0(s_0y), s_1(s_0y), s_2(s_0y), s_0(s_1y), s_1(s_1y), s_2(s_1y)$. However, these are not all distinct because of the relations among degeneracy maps: $s_1(s_0y) = s_0(s_0y), s_2(s_0y)=s_0(s_1y), s_2(s_1y)=s_1(s_1y)$. Thus, we have only three distinct generators: $s_0(s_0y), s_0(s_1y), s_1(s_1y)$. Denote these elements by $s_{00}y, s_{01}y, s_{11}y$.

This pattern continues. Using the degeneracy relations, we discover that the monoid $\fk_n$ is free on the generators 
$$s_{00\cdots 0}y, s_{00\cdots 01}y, s_{00\cdots 011}y,\dots ,s_{011\cdots 1}y, s_{11\cdots 1}y,$$ where the string in each subscript has length $n-1$; there are $n$ generators in total. To simplify notation we set 
$$\alpha_k^{(n)} = s_{00\cdots 011\cdots 1}y,\quad 1\le k\le n$$
where there are $(k-1)$ $1$'s in the subscript. The superscript $(n)$ denotes the dimension of the simplex. Note that $\alpha_1^{(1)}=y$. Because we will need a partial order on the simplices later, we order the generators in the obvious way:
$$\alpha_1^{(n)} < \alpha_2^{(n)} < \cdots <\alpha_n^{(n)}.$$
The boundary maps of these generators are as follows:
$$
d_i\bigl(\alpha_1^{(n)}\bigr) = \begin{cases}
				\alpha_1^{(n-1)} & 0\le i\le n-1 \\
				e & i=n
				\end{cases} \qquad
d_i\bigl(\alpha_n^{(n)}\bigr) = \begin{cases}
				e & i=0 \\
				\alpha_{n-1}^{(n-1)} & 1\le i\le n
				\end{cases}$$
and for $2\le k\le n-1$,
$$d_i\bigl(\alpha_k^{(n)}\bigr) = \begin{cases}
					\alpha_k^{(n-1)} & 0\le i\le n-k \\
					\alpha_{k-1}^{(n-1)} & n-k+1\le i\le n.
					\end{cases}$$
These maps are monoid homomorphisms and so the boundary of any $n$-simplex $\alpha_{i_1}^{(n)}\alpha_{i_2}^{(n)}\cdots \alpha_{i_j}^{(n)}$ is obtained as the product of the boundaries of each $\alpha_{i_\ell}^{(n)}$. Note that, given $\alpha_i^{(n)}$, the only $(n+1)$-simplex generators that have $\alpha_i^{(n)}$ as a face are $\alpha_i^{(n+1)}$ and $\alpha_{i+1}^{(n+1)}$.

We conclude this section with a few facts about $\fk$. Note that each generator $\alpha_i^{(n)}$, for $n\ge 2$ is degenerate since it is obtained via iterated degeneracies of the $1$-simplex $y$. Thus, in the geometric realization $|\fk|$, these simplices do not appear. However, products of these elements may or may not be degenerate. For example, a power of any single generator $\bigl(\alpha_i^{(n)}\bigr)^k$ is degenerate:
$$\bigl(\alpha_i^{(n)}\bigr)^k = \bigl(s_{00\cdots 01\cdots 1}(y)\bigr)^k = s_{00\cdots 01\cdots 1}(y^k).$$
We also have the following result.

\begin{proposition}\label{prop:topgen} If an $n$-simplex $\tau$ is nondegenerate, then the maximal $n$-simplex generator $\alpha_n^{(n)}$ occurs in the word $\tau$ at least once.
\end{proposition}

\begin{proof} Suppose that $\alpha_n^{(n)}$ does not appear in $\tau$ and write
$$\tau = \alpha_{i_1}\alpha_{i_2}\cdots \alpha_{i_\ell}$$ where each $i_j <n$ (we have dropped the superscript $(n)$ for notational convenience). Since each $i_j<n$, the sequence of degeneracies that creates $\alpha_{i_j}$ begins with an initial $s_0$. Since the degeneracy maps are monoid homomorphisms, we therefore see that 
$$\tau = s_0\bigl(\alpha_{j_1}^{(n-1)}\alpha_{j_2}^{(n-1)}\cdots\alpha_{j_\ell}^{(n-1)}\bigr),$$ where $j_k$ corresponds to the string $i_k$ with the initial $0$ removed. It follows that $\tau$ is a degenerate $n$-simplex.
\end{proof}

\section{Discrete Morse theory on simplicial sets}\label{sec:dmt}
Forman's discrete Morse theory \cite{forman} is formulated on CW-complexes, but an equivalent version for simplicial sets was developed earlier by Brown \cite{brown} (this is not as well-known as it should be). Brown called his procedure a {\em collapsing scheme} and he applied it to the classifying space of a monoid $M$ equipped with a {\em complete rewriting system}. We will employ Forman's language, but we will describe it Brown's terms.

Given a simplicial set $K$, if an $n$-simplex $x$ equals $d_iy$ for some $(n+1)$-simplex $y$, then we write $x<y$. Recall that $x$ is a {\em regular face} of $y$ if $x=d_iy$ for a unique index $i$.

\begin{definition}\label{def:morsematch} Let $K$ be a simplicial set and denote the set of $n$-simplices by $K_n$. A {\em Morse matching} on $K$ consists of the following.
\begin{enumerate}
\item For each $n$, there is a partition $K_n=A_n\cup B_n\cup C_n$, and
\item a bijection $c:A_n\to B_{n+1}$, such that if $c(\sigma)=\tau$ then $\sigma = d_i(\tau)$ for exactly one index $i$; that is, $\sigma$ is a regular face of $\tau$.
\end{enumerate}
These are required to satisfy the following conditions.
\begin{enumerate}
\item (acyclicity) If $\tau = c(\sigma)$, we write $\sigma\prec\tau$. Then there are no paths of the form
$$\sigma_0\prec\tau_0 >\sigma_1\prec\tau_1>\sigma_2\prec\cdots \prec \tau_k>\sigma_0,$$ where the $\sigma_i$ are distinct and $k>0$.
\item There are no infinite paths of the form 
$$\sigma_0\prec\tau_0 >\sigma_1\prec\tau_1>\sigma_2\prec\cdots >\sigma_k\prec \tau_k>\cdots$$ with each $\sigma_i\in A_n$.
\end{enumerate}
\end{definition}

The simplices in $C_n$ are called {\em critical}. Brown calls the elements of $A_n$ {\em redundant} and the corresponding elements of $B_{n+1}$ {\em collapsible}. He then proves the following.

\begin{proposition}[\cite{brown}, Prop. 1]\label{prop:brownprop}
 If $K$ has a collapsing scheme {\em (}Morse matching{\em )} then the geometric realization $X=|K|$ has a canonical quotient CW-complex $Y$, whose cells are in one-to-one correspondence with the critical simplices of $K$. Moreover, the quotient map $q:X\to Y$ is a homotopy equivalence mapping each open critical cell of $X$ homeomorphically onto the corresponding open cell of $Y$ and it maps each collapsible $(n+1)$-cell into the $n$-skeleton of $Y$. 
\end{proposition} 
 
 This corresponds to the classical theorems of smooth Morse theory, which assert that a manifold $M$, equipped with a Morse function $f:M\to\zr$, has the homotopy type of a CW-complex with one $i$-cell for each critical point of index $i$ for $f$.

A common tool for visualizing Morse matchings is via the {\em Hasse diagram} of $K$. This is the directed graph ${\mathcal H}_K$ whose vertices are the simplices of $K$ with a directed edge $\tau\to\sigma$, $\sigma\in K_n$ and $\tau\in K_{n+1}$, whenever $\sigma = d_i\tau$ for some $i$. Given a Morse matching, denote by $V$ the set of pairs $\{\sigma\prec\tau\}$ determined by the bijection $c$. Modify ${\mathcal H}_K$ by reversing the directed edge $\tau\to \sigma$ whenever $\{\sigma\prec\tau\}\in V$. The conditions on the matching then imply that the resulting directed graph is acyclic. Conversely, if one starts with ${\mathcal H}_K$ and reverses some edges corresponding to regular faces so that the resulting graph is acyclic, then the corresponding set of pairs determines a Morse matching on $K$. 

\section{A Morse matching on $\Omega S^2$}\label{sec:loopss2}
Our goal in this section is to describe a Morse matching on the simplicial monoid $\fk$ described in Section \ref{sec:fk}, which has the homotopy type of $\Omega S^2$. We will make use of Lampret's steepness pairing \cite{lampret} and for that we need to describe a partial order on the various $n$-simplices of $\fk$. 

Recall that we ordered the generators $\alpha_i^{(n)}$ of $\fk_n$ in the obvious way: $\alpha_i < \alpha_j$ if and only if $i<j$. The other elements of $\fk_n$ are words in these generators. Order the words of length $\ell$ lexicographically and then declare that all words of length $\ell$ are less than words of length $\ell+1$ for all $\ell\ge 1$. This is a total order on $\fk_n$. 

Now, given a simplicial set $X$, choose a total order $\le_k$ on the $k$-simplices $X_k$ and let $\le$ be the combined partial order on the vertices of the Hasse diagram ${\mathcal H}_X$. Visualize the elements of $X_k$ positioned vertically with $u$ above $v$ if and only if $u<v$, so that ${\mathcal H}_X$ has its vertices in stacks from $0$ on the left with the elements in $X_n$ to the right in increasing order of $n$.  Define the {\em steepness pairing} to be
$${\mathcal M}_{\le} = \biggl\{\sigma{\leftarrow}\tau \biggm| \textrm{$\sigma$ is a regular face of $\tau$}, \begin{array}{c}
{\forall\ \alpha\leftarrow\sigma: \alpha <\sigma} \\ 
{\forall\ \sigma\leftarrow\mu: \tau<\mu}
\end{array}\biggr\}.$$
Pictorially, these are the steepest edges connecting regular pairs of simplices. Note that there are edges connecting non-regular faces, and these must be considered when looking for steepest pairs.

\begin{proposition}[\cite{lampret}, Lemma 2.2]\label{prop:steep} Assume each $X_k$ is finite. Then ${\mathcal M}_{\le}$ is a Morse matching. Moreover, every Morse matching on $X$ is a subset of some steepness  pairing.
\end{proposition}

Lampret actually develops steepness pairings in the context of {\em algebraic} discrete Morse theory, but the argument used to prove Proposition \ref{prop:steep} applies in this context as well. A version of steepness pairings was also developed by Fasy, et al. \cite{fasy}, where they refer to {\em left-right parents} instead of steepest edges. 

We wish to apply this idea to the set $\fk$, but there is an obvious problem: each $\fk_n$ is (countably) infinite. We can get around this as follows. Fix a word length $\ell$. Recall that we have chosen the lexicographic order on the $n$-simplices of length $\ell$ for each $n$. Here is an important fact: given an $n$-simplex of word length $\ell$, all its faces have word length at most $\ell$. That is, word length can only decrease under the boundary homomorphisms. Now if we restrict our attention to the simplices of word length $\ell$, we can build the steepness pairing as above. In other words, our matching on $\fk$ will be the steepness pairing, but modified by pairing only simplices of the same word length, ignoring edges connecting words of different lengths. Also, we leave all degenerate simplices critical; this makes good sense as these disappear in the geometric realization anyway. We call this matching the {\em restricted steepness pairing} on $\fk$. 

\begin{theorem}\label{thm:restricted} The restricted steepness pairing is a Morse matching on $\fkmath$.
\end{theorem}

\begin{proof} We must show that there are no cycles and no infinite paths. For the cycles, note that within a given word length, there are no cycles by Proposition \ref{prop:steep}. That is, if we have a path of $n$- and $(n+1)$-simplices all of word length $\ell$, then the steepness pairing on these has no cycles. On the other hand, if a path of $n$- and $(n+1)$-simplices ever passes to a simplex of word length less than $\ell$, then it cannot come back since the pairing is restricted to words of the same length. Thus there cannot be any cycles. Also, there can be no infinite paths for the same reason--the boundary maps can only decrease word length and so any path must eventually terminate in words of length at least 1. Since there are only finitely many words in $\fk_n$ of length $\le \ell$ for any $\ell$, these paths must be finite.
\end{proof}

\section{Critical cells of the steepness pairing}\label{sec:crit}
Denote the set of pairs induced by the steepness pairing on $\fk$ by $V$. The steepness pairing is relatively inefficient in the sense that there are a lot of critical cells. As noted above, all degenerate simplices are left critical. But there are many others. For example, for a fixed simplex dimension $n$, if the word length $\ell$ of a simplex $\tau$ is much greater than $n$, then $\tau$ might have no regular faces, nor is it a regular face of an $(n+1)$-simplex. That is, $\tau$ cannot be in a pair in $V$. 

Our goal in this section is to describe a certain collection of critical cells that are homologically important. First observe that the unique $0$-simplex $e$ is critical since it is not a regular face of the $1$-simplex $y$. Since all $n$-simplices of word length $1$,  for $n\ge 2$, are degenerate, $y$ and all the higher generators are critical as well ($y$ cannot be matched with a generator of $\fk_2$). 

The first pairings occur in word length $2$. The $1$-simplex $y^2$ has steepest edge connecting it to $\alpha_1^{(2)}\alpha_2^{(2)}$; these form a pair in $V$. This leaves the $2$-simplex $\alpha_2^{(2)}\alpha_2^{(1)}$, and we will argue below that it does not pair with a word-length $2$ $3$-simplex. It is therefore left critical.

In word length $3$, the unique $1$-simplex $y^3$ pairs with $\alpha_1^2\alpha_2$ (superscript $(2)$'s omitted for notational clarity). There are six nondegenerate $2$-simplices of word length $3$, and some pair up with $3$-simplices as follows (dimension superscripts omitted):
$$ \{\alpha_1\alpha_2^2\prec\alpha_1\alpha_2\alpha_3\}, \; \{\alpha_2^2\alpha_1\prec \alpha_2\alpha_3\alpha_1\}, \;  \{\alpha_2\alpha_1\alpha_2\prec \alpha_2\alpha_1\alpha_3\}.$$
This leaves the $2$-simplices $\alpha_1\alpha_2\alpha_1$ and $\alpha_2\alpha_1^2$ critical. Also, no word-length $3$ $2$-simplex pairs with either $3$-simplex $\alpha_3\alpha_2^2$ or $\alpha_3\alpha_2\alpha_1$, and we will argue below that these do not pair with a word-length $3$ $4$-simplex. They are  therefore both left critical.

For each $k\ge 2$ consider the following simplices:
\begin{eqnarray*}
\sigma_{k} & =  & \alpha_{k}^{(k)}\alpha_{k-1}^{(k)}\cdots\alpha_1^{(k)}\\
\tau_{k} & = & \alpha_{k}^{(k)}\alpha_{k-1}^{(k)}\cdots \alpha_3^{(k)}\bigl(\alpha_2^{(k)}\bigr)^2 
\end{eqnarray*}
Set $\tau_0=\sigma_0=e$ and $\tau_1=\sigma_1=y$. 

\begin{theorem}\label{prop:sigmatau} The simplices $\sigma_{k}$ and $\tau_{k}$ are critical in the restricted steepness pairing.
\end{theorem}

\begin{proof} As noted above, the only generators that have $\alpha_i^{(n)}$ as a face are $\alpha_i^{(n+1)}$ and $\alpha_{i+1}^{(n+1)}$. Also, Proposition \ref{prop:topgen} tells us that nondegenerate $n$-simplices must contain the maximal generator $\alpha_n^{(n)}$. 

Consider the $k$-simplex $\sigma_{k} = \alpha_{k}\alpha_{k-1}\cdots\alpha_1$. We first show that it is not a regular face of any word length $k$ $(k+1)$-simplex. By the observations above, any nondegenerate $(k+1)$-simplex having $\sigma_{k}$ as a face must have $\alpha_{k+1}^{(k+1)}$ as its first element (it could not begin $\alpha_{k}\alpha_{k+1}$ since $d_0$ would have word length $k-1$, and that would be the only possible face that could equal $\sigma_k$). Proceeding left to right, we may then choose either $\alpha_i^{(k+1)}$ or $\alpha_{i+1}^{(k+1)}$ as part of a product to potentially have $\sigma_{k}$ as a face. Note, however, that we may never repeat a generator in consecutive indices, as that would lead to a square of a generator in the boundary. Thus, the only possible word-length $k$ $(k+1)$-simplices having $\sigma_{k}$ as a face are the various
$$\beta_s^{(k+1)} = \alpha_{k+1}^{(k+1)}\cdots \widehat{\alpha_{s}^{(k+1)}}\cdots \alpha_1^{(k+1)}$$  for $1\le s\le k$. Computing the boundary of $\beta_s$, we find that $\sigma_{k}$ is both $d_{k-s-1}(\beta_s)$ and $d_{k-s-2}(\beta_s)$ and hence $\sigma_{k}$ is not a regular face of any word length $k$ $(k+1)$-simplex. It therefore cannot pair up as part of a restricted steepness pair. 

The $0$-th and $k$-th faces of $\sigma_{k}$ are both the element $\sigma_{k-1}$, which has word length $k-1$. The regular faces of $\sigma_{k}$ are (all superscript indices are $(k-1)$):
$$\begin{array}{c}
\alpha_{k-1}^2\alpha_{k-2}\cdots\alpha_1 \\
\alpha_{k-1}\alpha_{k-2}^2\cdots\alpha_1 \\
\vdots \\
\alpha_{k-1}\alpha_{k-2}\cdots\alpha_1^2
\end{array}$$
Each of these is the regular face of a word length $k$ $k$-simplex lower in the lexicographic order. For example, $\alpha_{k-1}^2\alpha_{k-2}\cdots\alpha_1$ is a regular face of 
$$\alpha_{k-1}\alpha_k\alpha_{k-2}\cdots\alpha_1 < \sigma_{k}.$$ It follows that none of the regular faces of $\sigma_{k}$ forms a steepest pair with $\sigma_{k}$.  Thus, $\sigma_{k}$ is critical in the reduced steepness pairing.

The argument for $\tau_{k}$ is completely similar, but it is more subtle. The element $\tau_k$ is a regular face of precisely two word-length $k$ $(k+1)$-simplices:
$$\alpha_{k+1}\alpha_k\cdots\alpha_3\alpha_2\quad\textrm{and}\quad \alpha_{k+1}\alpha_k\cdots\alpha_2\alpha_3.$$ However, each of these simplices has a regular word-length $k$ face higher in the lexicographic order than $\tau_k$ and thus $\tau_k$ cannot be part of a restricted steepness pairing with any $(k+1)$-simplex. As with the $\sigma$'s, all the regular faces of $\tau_k$ are regular faces of simplices lower in the lexicographic order than $\tau_k$ and thus cannot form a pair with $\tau_k$. 
\end{proof}

 Recall that the homology of $\Omega S^2$ is $$H_r(\Omega S^2;\zz) \cong \zz, \quad r\ge 0.$$  The critical point of index $r$ for the energy function on $\Omega S^2$ corresponds to the geodesic tracing a great circle through the base point $r$ times. Let us be more explicit about the James map in this setting. It is convenient to use the {\em reduced suspension} here. If $(X,x_0)$ is a pointed space then $$\Sigma X = X\times [0,1]/(X\times \{0\}\cup X\times\{1\} \cup \{x_0\}\times [0,1]).$$ That is, collapse the subspaces $X\times \{0\}$ and $X\times \{1\}$ to points and then collapse the segment $\{x_0\}\times [0,1]$ to a point. We may view $X$ as embedded in $\Sigma X$ as the set of points $(x,1/2)$. Now if $x\in X$, define a loop $\gamma_x$ in $\Sigma X$ by $$\gamma_x(t) = [(x,t)],$$ where $[(x,t)]$ is the class of $(x,t)$ in $\Sigma X$. Note that this is in fact a loop in $\Sigma X$ since the points $(x,0)$ and $(x,1)$ are joined in the unreduced suspension by the segment $\{x_0\}\times [0,1]$, which is then collapsed to a point. In the case of $X=S^1$, we have $\Sigma X = S^2$; see Figure \ref{fig:jamesmap} for an example of a loop $\gamma_x$ (here, we take $x_0=1\in S^1$ as the basepoint). 
 
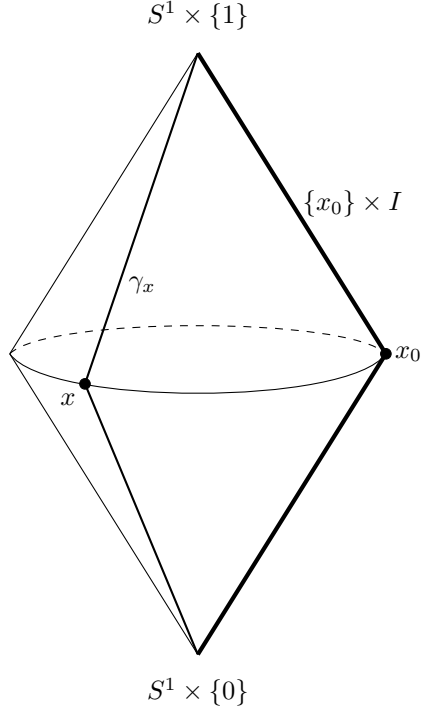
\begin{figure}
\begin{tikzpicture}
\draw[smooth,dashed] (-1,0) .. controls (-0.5,0.5) and (3.5,0.5).. (4,0);
\draw[smooth] (-1,0) .. controls (-0.5, -0.7) and (3.5,-0.7).. (4,0);
\draw (-1,0)--(1.5,4);
\draw (-1,0)--(1.5,-4);
\draw[ultra thick] (4,0)--(1.5,4) node[pos=0.5,right]{$\{x_0\}\times I$};
\draw[ultra thick] (4,0)--(1.5,-4);
\draw[thick] (0,-0.4)--(1.5,4) node[pos=0.3,right]{$\gamma_x$};
\draw[thick] (0,-0.4)--(1.5,-4);
\filldraw (0,-0.4) circle (2pt) node[anchor=north east]{$x$};
\filldraw (4,0) circle (2pt) node[anchor=west]{$x_0$};
\node at (1.5,-4.5) {$S^1\times\{0\}$};
\node at (1.5,4.5) {$S^1\times\{1\}$};
\end{tikzpicture}
\caption{\label{fig:jamesmap} A loop $\gamma_x$ in $\Sigma S^1=S^2$.}
\end{figure} 

The James map $J:X\to \Omega\Sigma X$ is then defined by $J(x)=\gamma_x$. Denote by $M(X,x_0)$ the free topological monoid on $X$. This is the quotient space
$$\coprod_{n\ge 0} X^n/\sim,$$ where $\sim$ is the equivalence relation generated by 
$$(x_1,\dots ,x_{i-1},x_0,x_{i+1},\dots ,x_n)\sim (x_1,\dots ,x_{i-1},x_{i+1},\dots ,x_n).$$ The multiplication is given by concatenation of words. We also have the space of {\em Moore loops} on $X$, $\Omega^M(X,x_0)$, defined as follows. Let $F(\zr,X)$ denote the space of all maps $\varphi:\zr\to X$ with the compact open topology. Let $\Omega^M(X,x_0)\subseteq F(\zr,X)\times [0,\infty)$ be the subspace of all pairs $(\varphi, r)$ such that $\varphi(0)=x_0$ and for which $\varphi(t)=x_0$ for $t\ge r$. The loop space $\Omega(X,x_0)$ is the subspace of pairs of the form $(\varphi,1)$ with $\varphi(t)=x_0$ for $t\ge 1$. It is straightforward to prove that $\Omega(X,x_0)$ is a deformation retract of $\Omega^M(X,x_0)$ (\cite{carlmilg}, Proposition 5.1.1).   Composing the map $J$ with the inclusion into $\Omega^M(\Sigma X,x_0)$ we obtain a map, still denoted $J$, but it does not carry $x_0$ to the identity. We can fix this by setting $\hat{X}=X\coprod [0,1]/\simeq$, where $\simeq$ is generated by $1\simeq x_0$ and then defining an extension $\hat{J}$ of $J$ to $\hat{X}$ by $\hat{J}(s) = (\varphi_{x_0},s)$, where $0\le s\le 1$ and $\varphi_{x_0}$ is the constant map with value $x_0$. Note that if $X$ is a CW-complex then $X$ and $\hat{X}$ are based homotopy equivalent. The map $\hat{J}$ is a pointed map if we take $0$ as the basepoint for $\hat{X}$. The map $\hat{J}$ then induces a monoid homomorphism 
$$\overline{J}:M(\hat{X},0)\to \Omega^M(\Sigma X,x_0).$$ James's theorem \cite{james} is that this map is a homotopy equivalence. The multiplication in $\Omega(\Sigma X,x_0)$ is given by concatenation of paths; this is only homotopy associative, of course, but this is why we expanded to $\Omega^M$, where this is not an issue. Note that this map is {\em not} a homeomorphism. Indeed, there are many more loops in $\Sigma X$ than are contained in the image of $\overline{J}$.

The simple geodesic on $\Sigma S^1$ corresponds to the image of the point $-1\in S^1$ under the map $\overline{J}$.  The critical point of index $r$, tracing this path $r$ times, then corresponds to the image of $\underbrace{(-1,-1,\dots ,-1)}_{r\,\textrm{times}}$ under $\overline{J}$. 

Milnor notes that the realization $|\fk|$ is homeomorphic to $M(\widehat{|K|},e)$ (\cite{milnor2}, Lemma 4). In fact, the product $(k_n,k_n',k_n'',\dots) \mapsto k_nk_n'k_n''\cdots \in (\fk)_n$ maps $K\times K\times\cdots\times K$ into $\fk$, and after taking realizations, this induces the homeomorphism. So we have a diagram of spaces
$$\Omega^M(\Sigma |K|) \stackrel{\overline{J}}{\longleftarrow} M(\widehat{|K|}) \longrightarrow |\fk|$$ in which the left map is a homotopy equivalence and the right map is a homeomorphism. 

Now take $K$ to be our simplicial model for $S^1$ and consider the element $\sigma_r\in (\fk)_r$ defined above: $\sigma_r = \alpha_r\alpha_{r-1}\cdots \alpha_2\alpha_1$ (superscript $(r)$'s omitted). This corresponds under the right map to the $r$-tuple $(\alpha_r,\alpha_{r-1},\dots ,\alpha_1)\in (K\times K\times\cdots\times K)_r$. Each $\alpha_i$ is a degenerate simplex in $K$, but the product element is nondegenerate. This gives us one way to build up an $r$-cell in $M(\widehat{|K|})$, and since each $\alpha_i$ is an iterated degeneracy of the unique nondegenerate $1$-simplex $y$, this is a reasonable candidate to serve as a representative of the critical geodesic of index $r$. A similar argument can be made for $\tau_r$, or other nondegenerate $r$-cells.

In the next section, we will study the action of the boundary map in $\zm_\bullet(\fk,V)$ on the elements $\sigma_r$ and $\tau_r$.

\section{Homology considerations}\label{sec:homology}
We first describe the boundary map in $\zm_\bullet(\fk,V)$. Forman's original definition \cite{forman} is given in terms of a discrete flow operator, and he then proves that this may be computed by counting gradient paths between critical cells. The latter description is analogous to what happens in the smooth context. Counting gradient paths is impractical for us, however, so we will use the original definition.

Let $V$ be a Morse matching on a cell complex $X$ and denote by $C_\bullet(X)$ the cellular chain complex. Define an inner product $\langle\; ,\;\rangle$ on each $C_n(X)$ by assuming the cells are orthogonal to each other. Define a map $V:C_n(X)\to C_{n+1}(X)$ by
$$V(\sigma) = -\langle \partial\tau,\sigma\rangle\tau$$ if $\{\sigma\prec\tau\}$ belongs to the matching $V$. If there is no such $\tau$, set $V(\sigma)=0$. Extend $V$ linearly.   The {\em discrete flow operator} is then the map $\Phi:C_n(X)\to C_n(X)$ defined by
$$\Phi(\sigma) = \sigma + \partial V(\sigma) + V(\partial\sigma);$$ that is, $\Phi = \textrm{id}+\partial V+V\partial$. Forman proves the following facts:
\begin{enumerate}
\item $\partial\Phi = \Phi\partial$.
\item Given $\sigma\in C_n(X)$ there exists $N$ such that $\Phi^N(\sigma)=\Phi^{N+1}(\sigma)=\cdots$. Denote this stable $n$-chain by $\Phi^{\infty}(\sigma)$. 
\item Denote the $\Phi$-invariant chains by $C_n^{\Phi}(X)$. Then the complex $C_\bullet^{\Phi}(X)$ has homology isomorphic to the singular homology of $X$.
\item Denote by $\zm_\bullet(X,V)$ the chain complex spanned by the critical cells of $V$. The stabilization map $\Phi^\infty:\zm_\bullet(X,V)\to C_\bullet^\Phi(X)$ is an isomorphism.
\item The boundary map $\tilde{\partial}$ in $\zm_\bullet$ may be computed as follows. If $c\in\zm_n(X,V)$ and $\sigma$ is a critical $(n-1)$-cell, then
$$\langle\tilde{\partial}c,\sigma\rangle = \langle \partial\Phi^\infty c,\sigma\rangle = \langle \Phi^\infty\partial c,\sigma\rangle.$$
\end{enumerate}

Recall the elements $\sigma_r$ and $\tau_r$ defined above. For $r\ge 3$, set 
\begin{eqnarray*}
\tilde{\sigma}_{r} & =  &\alpha_{r-1}\alpha_{r}\alpha_{r-2}\cdots \alpha_2\alpha_1 \\
\tilde{\tau}_{r} & = & \alpha_{r-1}\alpha_{r}\alpha_{r-2}\cdots\alpha_3\alpha_2^2
\end{eqnarray*}
i.e., transpose the first two generators of $\sigma$ or $\tau$. Note that this definition also makes sense for $\sigma_2$: $\tilde{\sigma}_2=\alpha_1\alpha_2$. 

\begin{proposition}\label{prop:stablefk}  We have the following stable chains
\begin{eqnarray*}
\Phi^\infty(\sigma_{r}) & = & \sigma_{r}-\tilde{\sigma}_{r}, \quad r\ge 2 \\
\Phi^\infty(\tau_{r}) & = & \tau_{r} - \tilde{\tau}_{r}, \quad r\ge 3.
\end{eqnarray*}
\end{proposition}

\begin{proof} Observe that if $\mu$ is a critical cell for $V$ then $\partial V(\mu)=0$.
Consider $\sigma_{r}$ first. We have the following:
\begin{eqnarray*}
\Phi(\sigma_{r}) & = & \sigma_{r} + \partial V(\sigma_{r}) + V(\partial\sigma_{r}) \\
 & = & \sigma_{r} + 0 + V(\sigma_{r-1} - \alpha_{r-1}^2\alpha_{r-2}\cdots \alpha_1 + \alpha_{r-1}\alpha_{r-2}^2\alpha_{r-3}\cdots\alpha_1 \\
 &     &  {} - \cdots + (-1)^{r-1} \alpha_{r-1}\cdots \alpha_2\alpha_1^2 + (-1)^{r}\sigma_{r-1}). \\
\end{eqnarray*}
Note that $V(\sigma_{r-1})=0$ since all the $\sigma_r$ are left unpaired by $V$. The $(r-1)$-simplex $\alpha_{r-1}^2\alpha_{r-2}\cdots \alpha_1$ is a regular face of $\tilde{\sigma}_{r}$ and it forms a steepness pair with it. The incidence number is $-1$ and so we find 
\begin{eqnarray*}
V(\alpha_{r-1}^2\alpha_{r-2}\cdots \alpha_1) & = & -\langle \partial\tilde{\sigma}_{r},\alpha_{r-1}^2\alpha_{r-2}\cdots \alpha_1\rangle\tilde{\sigma}_{r}  \\
 & = & \tilde{\sigma}_{r}.
\end{eqnarray*}
The remaining simplices in $\partial\sigma_{r}$ are also faces of various $r$-simplices, but they do not form steepness pairs with any of them since the other faces of these $r$-simplices lie both above and below the face in question in the lexicographic order (see an example in Section \ref{sec:apdx}). Thus we have
$$\Phi(\sigma_{r}) = \sigma_{r}-\tilde{\sigma}_{r}.$$
Now we compute
\begin{eqnarray*}
\Phi^2(\sigma_{r}) & = & \Phi(\sigma_{r}) - \Phi(\tilde{\sigma}_{r}) \\
 & = & \sigma_{r} - \tilde{\sigma}_{r} - \tilde{\sigma}_{r} -\partial V(\tilde{\sigma}_{r}) - V\partial(\tilde{\sigma}_{r}) \\
 & = & \sigma_{r}-2\tilde{\sigma}_{r} -V\partial(\tilde{\sigma}_{r}).
\end{eqnarray*}
The faces of $\tilde{\sigma}_{r}$ are similar to those of $\sigma_{r}$ and the same analysis applies: the face $\alpha_{r-1}^2\alpha_{r-2}\cdots \alpha_1$ is paired with $\tilde{\sigma}_{r}$ and the remaining faces cannot form steepness pairs with any $r$-simplex by order considerations. It follows that $V\partial(\tilde{\sigma}_{r}) = -\tilde{\sigma}_{r}$ and so
$$\Phi^2(\sigma_{r}) = \sigma_{r}-\tilde{\sigma}_{r}.$$ Thus, this is the stable chain $\Phi^\infty(\sigma_{r})$.
The argument for $\tau_{r}$ is completely similar.
\end{proof}

Note also that $\Phi^\infty(\tau_1) = \Phi^\infty(y) = y$ and $\Phi^\infty(e) = e$.

\begin{proposition}\label{prop:boundsigma} The following relations hold in the Morse complex $\zm_\bullet(\fkmath,V)${\em :}
\begin{eqnarray*}
\langle \tilde{\partial}\tau_{2k+1},\sigma_{2k}\rangle & = & 0 \\
\langle \tilde{\partial}\sigma_{2k},\tau_{2k-1}\rangle & = & 0.
\end{eqnarray*}
\end{proposition}

\begin{proof} This is a direct calculation. If $k\ge 1$, then
\begin{eqnarray*}
\langle \tilde{\partial}\tau_{2k+1},\sigma_{2k}\rangle & = & \langle \partial\Phi^\infty\tau_{2k+1},\sigma_{2k}\rangle \\
  & = & \langle \tau_{2k+1}-\tilde{\tau}_{2k+1},\sigma_{2k} \rangle \\
  & = & 0
\end{eqnarray*}
since $\sigma_{2k}$ is not a face of $\tau_{2k+1}$. Similarly, if $k\ge 2$, then
\begin{eqnarray*}
\langle \tilde{\partial}\sigma_{2k},\tau_{2k-1}\rangle & = & \langle \partial\Phi^\infty\sigma_{2k},\tau_{2k-1}\rangle \\
 & = & \langle \sigma_{2k}-\tilde{\sigma}_{2k},\tau_{2k-1}\rangle \\
 & = & 0.
\end{eqnarray*}
Finally, 
\begin{eqnarray*}
\langle \tilde{\partial}\sigma_2,\tau_1\rangle & = & \langle \partial\Phi^\infty\sigma_2,y\rangle \\
  & = & \langle \sigma_2-\tilde{\sigma}_2,y\rangle \\
  & = & 2 - 2 \\
  & = & 0.
\end{eqnarray*}
Clearly, $\langle y,e\rangle = 0$.
\end{proof}

\begin{remark} The $\sigma$'s do not behave well with respect to each other under the action of $\tilde{\partial}$. Indeed, each $\sigma_r$ is a nonregular face of $\sigma_{r+1}$ with incidence number $0$ or $2$ depending on the parity of $r$, while $\sigma_r$ {\em is} a regular face of $\tilde{\sigma}_{r+1}$ with incidence number 1. Thus we have
\begin{eqnarray*}
\langle\tilde{\partial}\sigma_{2k+1},\sigma_{2k}\rangle & = & 0 - 1 = -1 \\
\langle\tilde{\partial}\sigma_{2k},\sigma_{2k-1}\rangle & = & 2 - 1 = 1.
\end{eqnarray*}
\end{remark}

\begin{remark} The $\tau$'s do behave well. Indeed, $\tau_r$ is a regular face of both $\tau_{r+1}$ and $\tilde{\tau}_{r+1}$ with incidence number $1$, and so we have, for $r\ge 3$,
\begin{eqnarray*}
\langle\tilde{\partial}\tau_{r+1},\tau_r\rangle &= & 1-1 =0.
\end{eqnarray*}
\end{remark}

Now, Proposition \ref{prop:brownprop} implies that there is a quotient map 
$$q:|\fk|\to Y$$ where $Y$ is a CW-complex having one cell for each critical cell of the restricted steepness pairing. Note that $Y$ is an infinite complex. Denote by $Z$ the quotient of $Y$ obtained by collapsing all cells except for those in $\Sigma = \{\sigma_{2k},\tau_{2k+1}\}$. Then Proposition \ref{prop:boundsigma} implies that the homology of $Z$ is $$H_r(Z) \cong \zz,\quad r\ge 0.$$ Thus, we have identified a quotient of $|\fk|$ whose homology agrees with that of $|\fk|$, and it is in this sense that we may think of the $\sigma_r$ and $\tau_r$ as being homologically significant. The quotient map does not induce this isomorphism, however, since $Z$ is simply a wedge of spheres and so its cohomology ring is not isomorphic to the cohomology ring of $\Omega S^2$. A similar analysis applies to the set ${\mathcal T}=\{e,y,\sigma_2, \tau_r (r\ge 3)\}$.

One goal in discrete Morse theory is to find an optimal matching; that is, a matching with the minimal possible number of critical cells. As noted above, this is impossible in $\fk$ since simplices of large word length relative to their dimension must be critical. What we have done here is to identify a relatively small set of simplices that form cycles in the Morse complex and that seem to be reasonable candidates for homology generators. The complete story for computing the homology of $\zm_\bullet(\fk,V)$ is much more complicated, of course. We know the answer: $H_r\cong\zz$ for all $r\ge 0$, and the above analysis suggests that various $\sigma_r$ and $\tau_r$ are plausible candidates to be generating cycles. The problem, though, is that these simplices do not form a subcomplex, and their boundaries may involve other critical simplices in $V$. And there are other critical simplices that may interact with these as well. 

Since there is a unique critical vertex and a unique critical $1$-simplex, we know that $H_0(\zm_\bullet(\fk,V))\cong \zz$. Even computing $H_1$ is a challenge, however. We have the generator $y$ for $\zm_1(\fk,V)$ and it is a cycle. There are infinitely many critical $2$-simplices, however, having $y$ as a face. Proposition \ref{prop:boundsigma} tells us that $\langle\tilde{\partial}\sigma_2,y\rangle = 0$. But we also have a pair of critical word-length $2$ $2$-simplices: $\alpha_2\alpha_1^2$ and $\alpha_1\alpha_2\alpha_1$. These two simplices have the same boundary: $y^2-y^3+y$. It is convenient to use the formula $\langle\tilde{\partial}c,y\rangle = \langle\Phi^\infty\partial c,y\rangle$ to compute the action of $\tilde{\partial}$ in these cases.
Direct calculation shows that
\begin{eqnarray*}
\Phi^\infty(y^2) & = & 2y \\
\Phi^\infty(y^3) & = & 3y \\
\Phi^\infty(y) & = & y
\end{eqnarray*} 
and then we have
\begin{eqnarray*}
\langle\tilde{\partial}\alpha_2\alpha_1^2,y\rangle & = & \langle 2y-3y+y,y\rangle \\
       & = & 0 
\end{eqnarray*}
and
\begin{eqnarray*}
\langle\tilde{\partial}\alpha_1\alpha_2\alpha_1,y\rangle & = & \langle 2y-3y+y,y\rangle \\
       & = & 0.
\end{eqnarray*}

It is not practical to enumerate all the critical $2$-simplices, but it must be the case that they all map to $0$ under the Morse boundary $\tilde{\partial}$, and in fact we can prove this as follows.

\begin{lemma}\label{lem:stableyr} For each $r\ge 1$, we have $\Phi^\infty(y^r)=ry$.
\end{lemma}

\begin{proof} By induction on $r$, the case $r=1$ being clear since $\Phi(y)=y$. Assume the result holds for $r$. Note that $\{y^{r+1}\prec \alpha_1^{r}\alpha_2\}$ is a restricted steepness pair since $\alpha_1^{r}\alpha_2$ is the minimal element in the lexicographic order among the $2$-simplices of word-length $r+1$. We then compute
\begin{eqnarray*}
\Phi(y^{r+1}) & = & y^{r+1} + \partial V(y^{r+1}) + V(\partial y^{r+1}) \\
		& = & y^{r+1} + \partial \alpha_1^r\alpha_2 + 0 \\
		& = & y^{r+1} + y^r - y^{r+1} +y \\
		& = & y^r + y.
\end{eqnarray*}
Thus, by the inductive hypothesis
\begin{eqnarray*}
\Phi^\infty(y^{r+1}) & = & \Phi^\infty(y^r) + \Phi^\infty(y) \\
		&  = & ry + y \\
		& = & (r+1)y.
\end{eqnarray*}
\end{proof}

\begin{proposition}\label{prop:boundary2} The boundary map $\tilde{\partial}:\zm_2(\fkmath,V)\to \zm_1(\fkmath,V)$ is trivial.
\end{proposition}

\begin{proof} Any critical $2$-simplex of word length $n$ has the form $\sigma=\alpha_1^{i_1}\alpha_2^{i_2}\cdots\alpha_2^{i_{2k}}$, where $i_1$ and $i_{2k}$ could possibly equal $0$, and $\sum i_j = n$. The boundary of this simplex is
$$y^{i_1+i_3+\cdots+i_{2k-1}} - y^n + y^{i_2+i_4+\cdots i_{2k}}.$$
We then compute
\begin{eqnarray*}
\langle\tilde{\partial}\sigma,y\rangle & = & \langle\Phi^\infty\partial\sigma,y\rangle \\
		& = & \langle y^{i_1+i_3+\cdots+i_{2k-1}} - y^n + y^{i_2+i_4+\cdots i_{2k}},y\rangle \\
		& = & \biggl(\sum i_\textrm{odd} + \sum {i_\textrm{even}} - n\biggr) y \\
		& = & 0.
\end{eqnarray*}
\end{proof}

\begin{cor} $H_1(\zm(\fkmath,V))\cong\zz$. \hfill $\qed$
\end{cor}

\section{Appendix}\label{sec:apdx}
As promised in the previous section, here is an example of what happens when computing $\Phi(\sigma_6)$. The faces of $\sigma_6$ are
$$\begin{array}{c}
\sigma_5 \\ \alpha_5^2\alpha_4\alpha_3\alpha_2\alpha_1 \\ \alpha_5\alpha_4^2\alpha_3\alpha_2\alpha_1 \\
\alpha_5\alpha_4\alpha_3^2\alpha_2\alpha_1 \\
\alpha_5\alpha_4\alpha_3\alpha_2^2\alpha_1 \\
\alpha_5\alpha_4\alpha_3\alpha_2\alpha_1^2 \\
\sigma_5
\end{array}$$
As noted above, the face $\alpha_5^2\alpha_4\alpha_3\alpha_2\alpha_1$ forms a steepness pair with $\tilde{\sigma}_6$ since it is the face highest in the lexicographic order among all the faces of $\tilde{\sigma}_6$ and all $6$-simplices containing it as a face are higher in the order than $\tilde{\sigma}_6$. Consider the face $\alpha_5\alpha_4\alpha_3^2\alpha_2\alpha_1$. It is a regular face of the simplex $\alpha_6\alpha_5\alpha_3\alpha_4\alpha_2\alpha_1$, for example, but it is in the middle--that is, there are other faces higher in the order and lower in the order and so it cannot form a steepness pair. This holds true for any $6$-simplex having it as a face. This is a consequence of the combinatorial relations among the face and degeneracy maps in $\fk$ and it is a general phenomenon for all the faces of the various $\sigma_{r}$.

\bibliographystyle{plain}

\begin{thebibliography}{99}
\bibitem{brown} K.~Brown, {\em The geometry of rewriting systems: A proof of the Anick--Groves--Squier Theorem}, in {\em Algorithms and classification in combinatorial group theory}, G.~Baumslag and C.F.~Milller III, eds., MSRI Publications {\bf 23} (1992), 137--163.

\bibitem{carlmilg} G.~Carlsson and R.~J.~Milgram, {\em Stable homotopy and iterated loop spaces}, in  {\em Handbook of algebraic topology}, North-Holland, Amsterdam, (1995), 505--583.
    
\bibitem{fasy} B.~Fasy, B.~Holmgren, B.~McCoy, D.~Millman, {\em If you must choose among your children, pick the right one}, Proc. CCCG (2020), 336---345. (available online at \url{https://vga.usask.ca/cccg2020/papers/Proceedings.pdf})

\bibitem{forman} R.~Forman, {\em Morse theory for cell complexes}, Adv.~Math. {\bf 134} (1998), 90--145.


\bibitem{james} I.~James, {\em Reduced product spaces}, Ann. of Math. {\bf 62} (1955), 170--197.

\bibitem{lj} L.~Johnson, {\em Discrete Morse theory on the loop space of $S^2$}, doctoral dissertation, University of Florida, (2019).


\bibitem{lampret} L.~Lampret, {\em Chain complex reduction via fast digraph traversal}, preprint (2019). (available online at \url{https://arxiv.org/abs/1903.00783})

\bibitem{may} J.~P.~May, {\em Simplicial objects in algebraic topology,} University of Chicago Press, Chicago, (1992).

\bibitem{milnor} J.~Milnor, {\em Morse Theory}, Princeton University Press. Ann. Math. Studies {\bf 51}, Princeton Univ. Press, (1963).

\bibitem{milnor2} J.~Milnor, {\em On the construction} $\fk$, in {\em Algebraic Topology - A Student's Guide}, J.F. Adams, ed., London Math Soc. Lecture Notes {\bf 4}, Cambridge University Press, (1972), 118-136.


\end{thebibliography}

\end{document}